\documentclass[11pt]{article}
\usepackage{graphicx}
\usepackage{amsmath,latexsym,amsbsy,amssymb,amsthm, color,enumitem}

\newtheorem{theorem}{\textbf{Theorem}}[section]

\newtheorem{lemma}[theorem]{\textbf{Lemma}}
\newtheorem{corollary}[theorem]{\textbf{Corollary}}

\theoremstyle{definition}
\newtheorem{definition}[theorem]{\textbf{Definition}}

\theoremstyle{remark}

\newtheorem{remark}[theorem]{Remark}

\numberwithin{equation}{section}
\allowdisplaybreaks

\def\forall{\hbox{for all}~}

\def\implies{\Longrightarrow}

\def\v{\vskip 1em}

\def\Tilde{\widetilde}

\def\bega{\begin{array}}
\def\enda{\end{array}}
\def\begi{\begin{itemize}}
\def\endi{\end{itemize}}

\def\bel{\begin{equation}\label}
\def\eeq{\end{equation}}
\def\sqr#1#2{\vbox{\hrule height .#2pt
\hbox{\vrule width .#2pt height #1pt \kern #1pt
\vrule width .#2pt}\hrule height .#2pt }}

\title{\Large\textbf{A Stochastic Model of Optimal Debt Management and Bankruptcy}}

\bigskip

\author{\small Alberto Bressan$^{(*)}$, Antonio Marigonda$^{(**)}$, Khai T.~Nguyen$^{(***)}$, and Michele Palladino$^{(*)}$\\ \, \\
\small (*) Department of Mathematics, Penn State University\\
\small University Park, PA~16802, USA.\\
\small (**) Dipartimento di Informatica, Universit\`a di Verona, Italy\\
\small (***) Department of Mathematics, North Carolina State University, USA\\
\small Raleigh, NC~27695, USA.
\\ \, \\
\small e-mails:~ bressan@math.psu.edu,~antonio.marigonda@univr.it,\\
\small ~tnguye13@ncsu.edu,~mup26@psu.edu
}

\begin{document}

\maketitle

\begin{abstract} 
A problem of optimal debt management is modeled as a 
noncooperative game between a borrower and a pool of  lenders, 
in infinite time horizon with exponential discount. 
The  yearly income of the borrower is governed 
by a stochastic process.
When the debt-to-income ratio $x(t)$ reaches a given size $x^*$, 
bankruptcy instantly occurs. The interest rate charged by the
risk-neutral lenders is precisely 
determined in order to compensate for this possible loss 
of their investment.   

For a given bankruptcy  threshold $x^*$, 
existence and properties of optimal feedback 
strategies for the borrower are studied, in a stochastic 
framework  as well as in a limit deterministic setting.
The paper also analyzes how the expected total cost to the borrower
changes, depending on different values of $x^*$.
\end{abstract}

\section{Introduction}

We consider a problem of optimal debt management in infinite time horizon, 
modeled as a noncooperative game between a borrower and a pool of risk-neutral lenders. 
Since the debtor may go bankrupt, lenders charge a higher interest rate to offset the possible loss of part of their investment. 

In the models studied in \cite{BJ, BN}, the borrower has a fixed income, but large values of the debt determine a bankruptcy risk.
Namely, if at a given time $t$ the debt-to-income ratio $x(t)$ is too big, there is a positive probability that panic spreads among investors
and bankruptcy occurs within a short interval $[t, \, t+\varepsilon]$. This event is similar to a bank run.
Calling $T_B$ the random bankruptcy time, this means
\[\mathrm{Prob}\Big\{ T_B\in [t,\, t+\varepsilon]\,\Big|\,T_B>t\Big\}=\rho(x(t))\cdot\varepsilon+ o(\varepsilon).\]
Here the ``instantaneous bankruptcy risk'' $\rho(\cdot)$ is a given, nondecreasing function.

At all times $t$, the borrower must allocate a portion $u(t)\in [0,1]$ of his income to service the debt, i.e., 
paying back the principal together with the running interest. Our analysis will be mainly focused on the existence and properties 
of an optimal repayment strategy $u=u^*(x)$ in feedback form.

In the alternative model proposed by Nu\~{n}o and Thomas in \cite{NT}, the yearly income $Y(t)$ is modeled as a stochastic process:
\begin{equation}\label{eq:Yincome}dY(t)~=~\mu Y(t)\,dt+ \sigma Y(t)\,dW.\end{equation}

Here $\mu\geq 0$ is an exponential growth rate, while $W$ denotes Brownian motion on a filtered probability space.
Differently from \cite{BJ, BN}, in \cite{NT} it is the borrower himself that chooses when to declare bankruptcy.
This decision will be taken when the debt-to-income ratio reaches a certain threshold $x^*$, beyond which the burden of 
servicing the debt becomes worse than the cost of bankruptcy.

At the time $T_b$ when bankruptcy occurs, we assume that the borrower pays a fixed price $B$, while lenders recover a 
fraction $\theta(x(T_b))\in [0,1]$ of their outstanding capital.   Here $x\mapsto \theta(x)$ is a 
nondecreasing function of the debt size.   For example, the borrower may hold an amount $R_0$ of collateral 
(gold reserves, real estate$\ldots$) which will be proportionally divided among creditors if bankruptcy occurs.  In this case, 
when bankruptcy occurs each investor will receive a fraction 
\begin{equation}\label{eq:recover}\theta(x(T_b)) ~=~\max\left\{  \dfrac{R_0}{x(T_b)}\,,~1\right\}\end{equation}
of his outstanding capital. 

Aim of the present paper is to provide a detailed mathematical analysis of some models closely related to \cite{NT}. We stress that 
these problems are very different from a standard problem of optimal control. Indeed, the interest rate charged by lenders is not given a priori.
Rather, it is determined by the expected evolution of the debt at all future times. Hence it depends globally on the entire feedback control
$u(\cdot)$. A ``solution'' must be understood as a Nash equilibrium, where the strategy implemented by the borrower
represents the best reply to the strategy adopted by the lenders, and conversely.

Our main results can be summarized as follows.
\begin{itemize}
\item We first assume that value $x^*$ at which bankruptcy occurs is a priori given, and seek an optimal feedback 
control $u=u^*(x)$ which minimizes the expected cost to the borrower. 
For any value $\sigma\geq 0$ of the diffusion coefficient in \eqref{eq:Yincome}, we prove that the problem admits at least one  Nash equilibrium solution, in feedback form.  
In the deterministic case where $\sigma=0$, the solution can be constructed by concatenating solutions of a system of two ODEs, 
with terminal data given at $x=x^*$.
\item[]
\item We then study how the expected total cost
of servicing the debt together with the bankruptcy cost (exponentially discounted in time), are affected by different choices of~$x^*$.

Let $\theta(x^*)\in [0,1]$ be the salvage rate, i.e., the fraction of 
outstanding capital that will be payed back to lenders if bankruptcy occurs when the debt-to-income ratio is $x^*$. 
If
\begin{equation}\label{lbig}
\lim_{s\to +\infty}~\theta(s)\,s~=~+\infty,
\end{equation} then, 
letting $x^*\to +\infty$, 
the total expected cost to the borrower goes to zero. 
On the other hand, if
\begin{equation}\label{lsmall}
\lim_{s\to +\infty}~\theta(s)\,s~<~+\infty,
\end{equation}
then the total expected cost to the borrower remains uniformly positive as $x^*\to +\infty$.
\end{itemize}

We remark that the assumption \eqref{lsmall} is quite realistic. 
For example, if \eqref{eq:recover} holds, then $\theta(x^*)\,x^*=R_0$ for all
$x^*$ large enough.  We remark 
that \eqref{lsmall} rules out the possibility
of a \emph{Ponzi scheme}, where the old debt is serviced by initiating more and more new loans. Indeed, if \eqref{lsmall} holds, then 
such a strategy will cause the total debt to blow up to infinity in finite time.

The remainder of the paper is organized as follows.   In Section~2
we describe more carefully the model, deriving the equations
satisfied by the value function $V$ and the discounted bond price $p$.
In Sections 3 and 4 we construct equilibrium solutions in feedback
form, in the stochastic case ($\sigma>0$) and in the 
deterministic case ($\sigma=0$), respectively.
Finally, Sections 5 and 6 contain an analysis of how the expected cost to the borrower 
changes,  depending on the bankruptcy threshold $x^*$.

In the economics literature, some related models of debt and bankruptcy 
can be found in \cite{AG, AR, BJ, C, EG}.    A general introduction to Nash equilibria 
and differential games can be found in \cite{BO, B}.  For the 
basic theory of optimal control and viscosity solutions of Hamilton-Jacobi equations we refer to \cite{BCD, BPi}.

\section{A model with stochastic growth}

We consider a slight variant of the model in \cite{NT}. We denote by $X(t)$ the 
total debt of a borrower (a government, or a private company) 
at time $t$.
The annual income $Y(t)$ of the borrower is assumed to be a
random process, governed by the stochastic evolution equation \eqref{eq:Yincome}.

The debt is financed by issuing bonds. When an investor buys a bond of 
unit nominal value, he receives a continuous stream of payments with intensity
$(r+\lambda) e^{-\lambda t}$.  Here 
\begin{itemize}
\item $r$ is the interest rate payed on bonds, which we assume coincides with the discount rate, 
\item $\lambda$ is the  rate at which the borrower pays back the principal.
\end{itemize}
If no bankruptcy occurs, the payoff for an investor will thus be
\[\int_0^\infty e^{-r} (r+\lambda) e^{-\lambda t}\, dt=1.\]

In case of bankruptcy, a lender recovers only a fraction $\theta\in [0,1]$ of his outstanding capital. Here $\theta$ can depend on the total amount 
of debt at the time on bankruptcy. To offset this possible loss, the investor buys a bond with unit nominal value at a discounted price $p\in [0,1]$.
As in \cite{BN, NT}, at any time $t$ the value $p(t)$ is uniquely determined by the competition of a pool of risk-neutral lenders.

We call $U(t)$ the rate of payments that the borrower chooses to make to his creditors, at time $t$. If this amount is not enough to
cover the running interest and pay back part of the principal, new bonds are issued, at the discounted price $p(t)$.
The nominal value of the outstanding debt thus evolves according to
\begin{equation}\label{dX}
\dot X(t)=-\lambda X(t) + \dfrac{(\lambda + r) X(t) - U(t)}{p(t)}.
\end{equation}

The debt-to-income ratio is defined as $x= X/Y$. In view of \eqref{eq:Yincome} and (\ref{dX}), Ito's formula \cite{O, Shreve} yields the stochastic evolution equation
\begin{equation}\label{db}
dx(t)~=~\left[ \left(\dfrac{\lambda+r}{p(t)} -\lambda + \sigma^2 -\mu\right) x(t) - \dfrac{u(t)}{p(t)}\right]\, dt - \sigma \,x(t)\, dW.
\end{equation}
                                                                   
Here $u = U/Y$ is the portion of the total income allocated to pay 
for the debt.
Throughout the following we assume $r>\mu$.  We observe that, if $r<\mu$,
then the  borrower's income grows faster than the debt 
(even if no payment is ever made).
In this case, with probability one the debt-to-income ratio would approach 
zero as $t\to +\infty$.
\v
In this model, the borrower has two controls. At each time $t$ he can decide the portion $u(t)$ of the total income which he allocates to repay the debt. 
Moreover, he can decide at what time $T_b$ bankruptcy is declared. 
Throughout the following, we consider a control in feedback form, so that
\begin{equation}\label{ufbk}u~=~u^*(x)\qquad\qquad \hbox{for}\quad x\in [0,x^*],\end{equation}
while bankruptcy is declared as soon as $x(t)$ reaches the value $x^*$.
The bankruptcy time is thus the random variable
\begin{equation}\label{Tb}
T_b~\doteq~\inf\bigl\{ t>0\,;~~x(t)= x^*\bigr\}\,.
\end{equation}
At first, we  let $x^*$ be a given upper bound to the size of the debt.   
In a later section, we shall regard $x^*$ 
as an additional control parameter, chosen by the borrower in order to minimize his expected cost.

Given an initial size $x_0$ of the debt,
the total expected cost to the borrower, exponentially discounted in time, is thus computed as
\begin{equation}\label{scost}
J[x_0, u^*,x^*]~=~E\left[\int_0^{T_b} e^{-rt} L(u^*(x(t)))\, dt + e^{-r T_b} B\right]_{x(0)=x_0}.\end{equation}
Here $B$ is a large constant, accounting for the bankruptcy cost,
while $L(u)$ is the instantaneous cost for the borrower to implement the control $u$. 
In the following we shall assume 
\v
{\bf (A)
} {\it The cost function $L$ is twice continuously differentiable for $u\in [0,1[\,$
and satisfies
\begin{equation}\label{Lprop}L(0)=0,\qquad L'>0, \quad
 L''>0,\qquad\quad \lim_{u\to 1-} L(u)~=~+\infty.\end{equation}}

For example, one may take
\[L(u)=c\,\ln\dfrac{1}{1-u},\qquad \textrm{ or }\qquad L(u)=\dfrac{c u}{(1-u)^\alpha}\,,\]
for some $c,\alpha>0$.

Fix $x^*>0$.
For a given initial debt $x(0)=x_0\in [0, x^*]$, we define the corresponding value function as
\begin{equation}\label{V}
V(x_0)~=~\inf_{u^*(\cdot)}
~J[x_0, u^*, x^*]\,.\end{equation} 
Under the assumptions {\bf (A)
} we have
\begin{equation}\label{bd4}
V(0)~=~0\,,\qquad
\qquad V(x^*)~=~B\,.
\end{equation}
Denote by
\begin{equation}\label{H1}
H(x,\xi,p)~\doteq~\min_{\omega\in [0,1]}
~\left\{ L(\omega) - \dfrac{\xi}{p}\, \omega\right\}+\left(\dfrac{\lambda+r}{p} -\lambda + \sigma^2 -\mu
\right) x\, \xi\,
\end{equation}
the Hamiltonian associated to the dynamics (\ref{db}) and the cost function 
$L$ in (\ref{scost}). Notice that, as long as $p>0$, the function $H$
is differentiable with Lipschitz continuous derivatives w.r.t.~all arguments. 

By a standard  arguments,   
the value function $V$ provides a solution to the second order ODE
\begin{equation}\label{EqV}
r V(x)~=~H\bigl(x,V'(x),p(x)\bigr)+\dfrac{(\sigma x)^2}{2}V''(x)\,,
\end{equation}
with boundary conditions (\ref{bd4}).
As soon as the function $V$ is determined, the 
optimal feedback control is recovered by 
\[u^*(x)=\underset{\omega\in [0,1]}{\mathrm{argmin}}\left\{ L(\omega) - \dfrac{V'(x)}{p(x)}\, \omega\right\}.\]
By {\bf (A)
} this yields
\begin{equation}\label{u^*}
u^*(x)~=~\begin{cases}
0&\qquad\mathrm{if}\qquad\displaystyle\dfrac{V'(x)}{p(x)}~\leq~L'(0) \,,\\[4mm]
\displaystyle (L')^{-1}\left(\dfrac{V'(x)}{p(x)}\right)&\qquad \mathrm{if}\qquad \displaystyle\dfrac{V'(x)}{p(x)}~>~L'(0)\,.\end{cases}
\end{equation}
Assuming that lenders are risk-neutral, the discounted bond price $p$ is determined by
\begin{equation}\label{p}
p(x_0)~=~ E\Big[\int_0^{T_b}(r+\lambda)e^{-(r+\lambda)t}dt+
e^{-(r+\lambda)T_b}\theta(x^*)\Big]_{x(0)=x_0}\,,
\end{equation}
where $\theta$ denotes the salvage rate.  In other words, if bankruptcy
occurs when the debt-to-income ratio is $x^*$, then investors receive a fraction $\theta(x^*)\in [0,1]$ of the nominal value of their holding. Notice that 
the random variable $T_b$ in (\ref{Tb}) now 
depends on the initial state $x_0$, the threshold $x^*$, and on the feedback control 
$u^*(\cdot)$.

By the Feynman-Kac formula, $p(\cdot)$  satisfies the equation
\begin{equation}\label{pHJ}
(r+\lambda)(p(x)-1)~=~\left[ \left(\dfrac{\lambda+r}{p(x)} -\lambda + \sigma^2 -\mu
\right) x - \dfrac{u^*(x)}{p(x)}\right]\cdot p'(x)+ \dfrac{(\sigma x)^2}{2} p''(x),\end{equation}
with boundary values
\begin{equation}\label{pbv}p(0)~ =~ 1,\qquad\qquad p(x^*) ~=~\theta(x^*)\,.\end{equation}
Combining (\ref{EqV}) and (\ref{pHJ}), 
we are thus led to the system of second order ODEs
\begin{equation}\label{ode0}
\begin{cases}
rV(x)&= ~\displaystyle H\bigl(x,V'(x),p(x)\bigr) + \dfrac{(\sigma x)^2}{2}\cdot V''(x) \,,\\[6mm]
(r+\lambda)(p(x)-1)&=~\displaystyle H_{\xi}\bigl(x,V'(x),p(x)\bigr)\cdot p'(x)+\dfrac{(\sigma x)^2}{2}\cdot p''(x)\,,
\end{cases}
\end{equation}
with the boundary conditions
\begin{equation}\label{bdc}
\begin{cases}V(0)&=~0,\\ \\ V(x^*)&=~B,\end{cases}\qquad\qquad
\begin{cases}p(0)&=~1,\\ \\ p(x^*)&=~\theta(x^*).\end{cases}
\end{equation}

We close this section by collecting some useful properties of the Hamiltonian function.

\begin{lemma}\label{lemma:Hprop}
Let the assumptions {\bf (A)}
 hold. 
Then, for all $\xi\geq 0$ and $p\in [0,1]$,  the function $H$ in (\ref{H1})
satisfies
\begin{equation}\label{Hb1}
\left(\frac{(\lambda+r) x-1}{p}+(\sigma^2-\lambda-\mu)x\right)\xi ~\leq~H(x,\xi,p)~\leq~\left(\dfrac{\lambda+r}{p} -\lambda + \sigma^2 -\mu\right) x\xi,
\end{equation}
\begin{equation}\label{Hb2}
\frac{(\lambda+r) x-1}{p}+(\sigma^2-\lambda-\mu)x~\leq~H_\xi(x,\xi,p)~\leq~\left(\dfrac{\lambda+r}{p} -\lambda + \sigma^2 -\mu\right) x.
\end{equation}
Moreover, for every $x,p>0$ 
the map $\xi\mapsto H(x,\xi,p)$ is concave down and satisfies
\begin{align}
\label{H00}H(x,0,p)&~=~0,\\ 
\label{H0x}H_\xi(x,0,p)&~=~\left(\dfrac{\lambda+r}{p} -\lambda + \sigma^2 -\mu\right) x\,,\\
\label{limH}\lim_{\xi\to +\infty} H(x,\xi,p)&~=~
\begin{cases}
-\infty,&\textrm{ if }~~\dfrac{1}{p}>\left(\dfrac{\lambda+r}{p} -\lambda + \sigma^2 -\mu\right) x\,,\\[4mm]
+\infty,&\textrm{ if }~~\dfrac{1}{p}\leq\left(\dfrac{\lambda+r}{p} -\lambda + \sigma^2 -\mu\right) x\,.
\end{cases}
\end{align}
\end{lemma}

\begin{proof}
\begin{enumerate}[leftmargin=*,label=\textbf{\arabic*.}] 
\item[]
\item Since $H(x,\cdot, p)$ is defined as the infimum of a family of affine functions, it is concave down.  
We observe that
(\ref{H1}) implies
\begin{equation}\label{H5}
H(x,\xi,p)~=~\left(\dfrac{\lambda+r}{p} -\lambda + \sigma^2 -\mu
\right) x\xi\qquad\qquad \hbox{if}\quad 0\leq\xi\leq p L'(0).\end{equation}
This yields the identities (\ref{H00})-(\ref{H0x}).
\item[]
\item Taking $\omega=0$ in (\ref{H1}) we obtain the upper bound in (\ref{Hb1}).
By the concavity property, the map $\xi\mapsto H_\xi(x,\xi,p)$ is non-increasing.
Hence (\ref{H0x}) yields the upper bound in (\ref{Hb2}).
\item[]
\item Since $L(w)\geq 0$ for all $w\in [0,1]$, we have
\[
H(x,\xi,p)~\geq~\min_{w\in [0,1]}\left\{-\dfrac{\xi}{p}w\right\}+\left(\dfrac{\lambda+r}{p} -\lambda + \sigma^2 -\mu
\right) x\, \xi\,
\]
and obtain the lower bound in (\ref{Hb1}). On the other hand, using the optimality condition, one computes from (\ref{H1}) that 
\begin{equation}\label{HXX}
H_{\xi}(x,\xi,p)~=~\frac{(\lambda+r) x-u^*(\xi,p)}{p}+(\sigma^2-\lambda-\mu)x
\end{equation}
where 
\[
u^*(\xi,p)~=~\underset{\omega\in [0,1]}{\mathrm{argmin}}~\left\{ L(\omega) - \dfrac{\xi}{p}\, \omega\right\}~=~(L')^{-1}\left(\dfrac{\xi}{p}\right)~<~1\,.
\]
Observe that, as $\xi\to +\infty$, one has $u^*(\xi,p)\to 1$ in (\ref{HXX}).
The non-increasing property of the map $\xi\to H_{\xi}(x,\xi,p)$ yields  the lower bound in (\ref{Hb2}).
\item[]
\item To prove (\ref{limH}) we observe that, in the first case, 
there exists $\omega_0<1$ such that
\[\dfrac{\omega_0}{p}~>~\left(\dfrac{\lambda+r}{p} -\lambda + \sigma^2 -\mu \right) x.\]
Hence, letting $\xi\to +\infty$ we obtain
\[\lim_{\xi\to +\infty} H(x,\xi,p)~\leq~
~\lim_{\xi\to +\infty}~\left[ L(\omega_0) - \dfrac{\omega_0}{p} \xi+\left(\dfrac{\lambda+r}{p} -\lambda + \sigma^2 -\mu
\right) x\, \xi\right]~=~-\infty\,.\]
To handle the second case, we observe that, for $\xi>0$ large,
the minimum in (\ref{H1}) is attained at the unique point $\omega(\xi)$
where 
$L'(\omega(\xi))~=~{\xi/p}$. Hence $\displaystyle\lim_{\xi\to +\infty} \omega(\xi)= 1$ and
\begin{align*}
\lim_{\xi\to +\infty}H(x,\xi,p)&\displaystyle=~\lim_{\xi\to +\infty} \left[ L(\omega(\xi)) - \dfrac{\omega(\xi)}{p} \xi+\left(\dfrac{\lambda+r}{p} -\lambda + \sigma^2 -\mu
\right) x\, \xi\right]\\[4mm]
&\displaystyle\geq~\lim_{\xi\to +\infty} L(\omega(\xi))~=~+\infty.
\end{align*}
\end{enumerate}
\end{proof}

\section{Existence of solutions} 

Let $x^*>0$ be given.
If a solution $(V,p)$ to the boundary value problem (\ref{ode0})-(\ref{bdc}) is found, 
then 
the feedback control $u=u^*(x)$ defined at (\ref{u^*}) and the function $p=p(x)$
provide an equilibrium solution to the debt management problem.
In other words, for every initial value $x_0$ of the debt, the following holds.
\begin{itemize}
\item[(i)]  Consider the stochastic dynamics (\ref{db}),
with $u(t) = u^*(x(t))$ and  $p(t) = p(x(t))$.   
Then for every $x_0\in [0,x^*]$ the identity (\ref{p})
holds, where $T_b$ is the random bankruptcy time defined at (\ref{Tb}) .  

\item[(ii)] Given the discounted price $p=p(x)$, for every initial data 
$x(0)=x_0\in [0,x^*]$ the feedback 
control $u=u^*(x)$ is optimal for the stochastic optimization problem
\begin{equation}\label{smin}
\hbox{minimize:}~~E\left[\int_0^{T_b} e^{-rt} L(u^*(x(t)))\, dt + e^{-r T_b} B\right].\end{equation}
with stochastic dynamics (\ref{db}), where $p(t) = p(x(t))$.
\end{itemize}

To construct a solution to the system
(\ref{ode0})-(\ref{bdc}),
we consider the auxiliary parabolic system
\begin{equation}\label{par1}
\begin{cases}
V_t (t,x)&= ~\displaystyle -rV(t,x)+ H\bigl(x,V_x(t,x),p(t,x)\bigr) + \dfrac{(\sigma x)^2}{2}\cdot V_{xx}(t,x) \,,\\[4mm]
p_t(t,x)&=~(r+\lambda)(1-p(t,x))+ \displaystyle H_{\xi}\bigl(x,V_x(t,x),p(t,x)\bigr)\cdot p_x(t,x)+\dfrac{(\sigma x)^2}{2}\cdot p_{xx}(t,x)\,,
\end{cases}
\end{equation}
with boundary conditions (\ref{bdc}). 
Following \cite{A}, the main idea is to construct a compact, convex set of functions 
$(V,p): [0, x^*]\mapsto [0,B]\times [\theta(x^*),1]$
which is positively invariant for the parabolic evolution problem.
A topological technique will then yield the existence of a steady state,
i.e.~a solution to (\ref{ode0})-(\ref{bdc}). 

\begin{theorem}\label{thm:existence}
In addition to {\bf (A)},  assume that $\sigma>0$ and 
$ \theta(x^*)>0$. 
Then the system of second order ODEs (\ref{ode0}) with boundary conditions 
(\ref{bdc}) admits a $\mathcal{C}^2$ solution $(\overline{V},\bar{p})$, 
such that $\overline{V}:[0,x^*]\to [0,B]$ is increasing and 
$\bar{p}:[0,x^*]\to [\theta(x^*),1]$ is decreasing. 
\end{theorem}
\begin{proof}
\begin{enumerate}[leftmargin=*,label=\textbf{\arabic*.}] 
\item[]
\item For any $\varepsilon>0$, consider the parabolic system
\begin{equation}\label{parV}
V_t = ~\displaystyle -rV+ H(x,V_x,p) + \left(\varepsilon+\dfrac{(\sigma x)^2}{2}\right)V_{xx}\,,\qquad \qquad 
\begin{cases}V(0)&=~0,\cr\cr V(x^*)&=~B,\end{cases}
\end{equation}
\begin{equation}\label{parp}
p_t~=~(r+\lambda)(1-p)+ \displaystyle H_{\xi}(x,V_x,p) p_x
+ \left(\varepsilon+\dfrac{(\sigma x)^2}{2}\right) p_{xx}\,,\qquad\qquad\begin{cases}p(0)&=~1,\\ \\
 p(x^*)&=~\theta(x^*)\,.\end{cases}\end{equation}
obtained from (\ref{par1}) by adding the terms $\varepsilon V_{xx}$, 
$\varepsilon p_{xx}$ on the right hand sides. For any $\varepsilon>0$,
this renders the system uniformly parabolic, also in a neighborhood of $x=0$.
\item[]
\item Adopting a semigroup notation, let
$t\mapsto (V(t), p(t))= S_t(V_0, p_0)$ be the solution of the system
(\ref{parV})-(\ref{parp}), with initial data
\begin{equation}\label{idvp}
V(0,x)=V_0(x),\qquad p(0,x)=p_0(x).\end{equation}
Consider the
closed, convex set of functions 
\begin{equation}\label{Ddef}
\mathcal{D}~=~\Big\{ (V,p):[0,x^*]\mapsto [0,B]\times [\theta(x^*),1]~;~~~V,p\in \mathcal{C}^2,\,V_x\geq 0,\,p_x\leq 0,\,\textrm{ and (\ref{bdc}) holds}\Big\}\,.
\end{equation}
We claim that the above domain is positively invariant under the semigroup $S$, namely
\begin{equation}\label{invar}
S_t(\mathcal{D})~\subseteq~\mathcal{D} \qquad \qquad\forall t\geq 0\,.
\end{equation}
Indeed, consider the constant functions
\[\begin{cases}V^+(t,x)&=~B,\\ \\ V^-(t,x)&=~0,\end{cases}\qquad\qquad\begin{cases}p^+(t,x)&=~1,\\ \\  p^-(t,x)&=~\theta(x^*)\,.\end{cases}\] 
Recalling (\ref{H00}), one easily checks that
$V^+$ is a supersolution and $V^-$ is a subsolution of the scalar
parabolic problem (\ref{parV}).  Indeed
$$-rV^++ H(x,V^+_x,p) + \left(\varepsilon+\dfrac{(\sigma x)^2}{2}\right)
V^+_{xx}~\leq~0,\quad\qquad V^+(t,0)\geq 0,\qquad V^+(t,x^*)\geq B.$$ 
$$-rV^-+ H(x,V^-_x,p) + \left(\varepsilon+\dfrac{(\sigma x)^2}{2}\right)
V^-_{xx}~\geq~0,\quad\qquad V^-(t,0)\leq 0,\qquad V^-(t,x^*)\leq B.$$ Similarly, $p^+$ is a supersolution and $p^-$ is a subsolution of the scalar parabolic problem (\ref{parp}). 

This proves that, if the initial data $V_0,p_0$ in (\ref{idvp}) take values in the box
$[0, B]\times [\theta(x^*),\, 1]$, then for every $t\geq 0$ the solution of the system
(\ref{parV})-(\ref{parp}) will satisfy
\begin{equation}\label{comp2}
0~\leq~V(t,x)~\leq~B,\qquad\qquad \theta(x^*)~\leq~p(t,x)~\leq~1,\end{equation} 
for all $x\in [0,x^*]$. In turn, this implies
\begin{equation}\label{b++}\begin{cases}
V_x(t,0)&\geq~0,\\ \\ V_x(t,x^*)&\geq ~0,\end{cases}\qquad\qquad \begin{cases}p_x(t,0)&\leq~0,\\ \\
 p_x(t,x^*)&\leq ~0\,.\end{cases}
\end{equation}
\item[]
\item Next, we prove that the monotonicity properties of $V(t,\cdot)$ and 
$p(t,\cdot)$ are preserved in time.  
Differentiating w.r.t.~$x$
one obtains
\begin{equation}\label{Vxt} V_{xt}~=~- r V_x + H_x + H_\xi V_{xx} + H_pp_x +
\sigma^2 x V_{xx} +\left(\varepsilon+\dfrac{(\sigma x)^2}{2}\right) 
V_{xxx}\,,\end{equation}
\begin{equation}\label{pxt}p_{xt}~=~-(r+\lambda) p_x + \left(\dfrac{d}{dx} H_\xi(x,V_x,p)\right) p_x
+ H_\xi p_{xx} + 
\sigma^2 x p_{xx} +\left(\varepsilon+\dfrac{(\sigma x)^2}{2}\right) 
p_{xxx}\,.\end{equation}
By (\ref{H00}),  for every $x,p$ one has $H_x(x,0,p) = H_p(x,0,p)=0$.
Hence $V_x\equiv 0$ is a subsolution of (\ref{Vxt}) and 
$p_x\equiv 0$ is a supersolution of (\ref{pxt}).
In view of (\ref{b++}), we obtain
$$p_x(t,x)~\leq~0~\leq V_x(t,x)\qquad
\qquad\forall ~~t\geq 0,~~x\in [0,x^*].$$
This concludes the proof that
the set $\mathcal{D}$ in (\ref{Ddef})
is positively invariant for the system (\ref{parV})-(\ref{parp}).
\item[]
\item Thanks to the bounds (\ref{Hb1})-(\ref{Hb2}), we
can now apply Theorem~3 in \cite{A} and obtain the existence of a steady state 
$( V^\varepsilon, p^\varepsilon)\in \mathcal{D}$ for the system 
(\ref{parV})-(\ref{parp}).  

We recall the main argument in \cite{A}. For every $T>0$ the map
$(V_0, p_0)\mapsto S_T(V_0, p_0)$ is a compact transformation
of the closed convex domain $\mathcal{D}$ into itself. By Schauder's 
theorem it has a fixed point.  This yields a periodic solution
of the parabolic system (\ref{parV})-(\ref{parp}), with period $T$.
Letting $T\to 0$, one obtains a steady state.
\item[]
\item It now remains to derive a priori estimates on this stationary
solution, which will allow to take the limit as $\varepsilon\to 0$.
Consider any solution to
\begin{equation}\label{VPE}\begin{cases}\displaystyle
-rV+ H(x,V',p) + \left(\varepsilon+\dfrac{(\sigma x)^2}{2}\right)
V''&=~0\,,\\[4mm]
(r+\lambda)(1-p)+ \displaystyle H_{\xi}(x,V',p) p'
+ \left(\varepsilon+\dfrac{(\sigma x)^2}{2}\right) p''&=~0\,,\end{cases}\end{equation}
with $V$ increasing, $p$ decreasing, and satisfying the boundary conditions (\ref{bdc}).

By the properties of $H$ derived in Lemma \ref{lemma:Hprop}, we can find $\delta>0$ small enough 
and $\xi_0>0$ such that the following implication holds:
\[x\in [0,\delta],~~~p\in [\theta(x^*),1],~~~ \xi\geq\xi_0  \qquad\implies\qquad
H(x,\xi,p)\leq 0\,.\]
As a consequence, if $V'(x)>\xi_0$ for some $x\in [0,\delta]$, then 
the first equation in (\ref{VPE}) implies $V''(x)\geq 0$.
We conclude that either $V'(x)\leq \xi_0$ for all $x\in [0,x^*]$, or else
$V'$ attains its maximum on the subinterval $[\delta, x^*]$.

By the intermediate value theorem, there exists a point $\hat x\in [\delta, x^*]$ 
where 
\begin{equation}\label{Vxh}V'(\hat x)~=~\dfrac{V(x^*)- V(\delta)}{x^*-\delta}~\leq~ \dfrac{B}{x^*-\delta}\,.\end{equation}
By (\ref{Hb1}), the derivative $V'$ satisfies a differential inequality of the form
\begin{equation}\label{Vtt}|V''|~\leq~c_1 |V'|+c_2\,,\qquad\qquad x\in [\delta, x^*]\,.\end{equation}
for suitable constants $c_1, c_2$.
By Gronwall's lemma, from the differential inequality (\ref{Vtt}) and the 
estimate (\ref{Vxh}) 
one obtains a uniform bound on $V'(x)$, for all $x\in [\delta, \hat x]\cup[\hat x, x^*]$.
\item[]
\item Similar arguments apply to $p'$.  By (\ref{Hb2}), the term $H_\xi(x,V',p)$ in 
(\ref{VPE}) is uniformly bounded. 
For every $\delta>0$, by (\ref{VPE}) shows that $p'$ satisfies a linear ODE
whose coefficients remain bounded on $[\delta, x^*]$, uniformly 
w.r.t.~$\varepsilon$. This yields the bound
$$|p'(x)|~\leq~C_\delta\qquad\qquad \forall x\in [\delta, x^*]$$
for some constant $C_\delta$, uniformly valid as $\varepsilon\to 0$.

To make sure that,
as $\varepsilon\to 0$, the limit satisfies the boundary value $p(0)=1$.
one needs to provide a lower bound on $p$ also in a neighborhood of $x=0$, independent of $\varepsilon$. Introduce the constant
$$\gamma~\doteq~\min\left\{ 1\,, ~(r+\lambda)
\left(\dfrac{\lambda+r}{ \theta(x^*)} -\lambda + \sigma^2 -\mu
\right)^{-1}\right\}.$$
Then define
$$p^-(x)~\doteq~1- c x^\gamma,$$
choosing $c>0$ so that $p^-(x^*)=\theta(x^*)$.
We claim that the convex function $p^-$
is a lower solution of the second equation in (\ref{VPE}). Indeed, by 
(\ref{VPE}), one has
$$(r+\lambda) c x^\gamma- H_\xi(x, V', p)\,c \gamma x^{\gamma-1}
~\geq~\left[(r+\gamma) - \left(\dfrac{\lambda+r}{ \theta(x^*)} -\lambda + \sigma^2 -\mu
\right)\gamma\right] cx^\gamma~\geq~0.$$
\item[]
\item Letting $\varepsilon\to 0$, we now consider a sequence 
$(V^{\varepsilon}, p^{\varepsilon})$ of solutions to (\ref{VPE}) with 
boundary conditions (\ref{bdc}).  Thanks to the previous estimates, 
the functions $V^{\varepsilon}$ are uniformly Lipschitz continuous on $[0, x^*]$, 
while the functions $p^\varepsilon$ are Lipschitz continuous
on any subinterval $[\delta,x^*]$ and satisfy 
$$p^-(x)~\leq ~p^\varepsilon(x)~\leq ~1\qquad\qquad\forall~x\in [0, x^*],~~\varepsilon >0.$$
By choosing a suitable subsequence, we achieve the uniform convergence
$(V^{\varepsilon}, p^{\varepsilon})\to (V,p)$, where $V,p$ are twice continuously differentiable
on the open interval $\,]0, x^*[$, and satisfy the boundary conditions (\ref{bdc}).
\end{enumerate}
\end{proof}

\section{The deterministic case}

If $\sigma=0$, then the stochastic equation (\ref{db}) 
reduces to the deterministic control system 
\begin{equation}\label{detode}
\dot x~=~\left(\dfrac{\lambda+r}{p} -\lambda -\mu
\right) x - \dfrac{u}{p} \,.\end{equation}

We then consider the deterministic Debt Management Problem.
\begin{itemize}
\item[{\bf (DMP)}]  {\it Given an initial value $x(0)=x_0\in [0, x^*]$ of the debt, 
minimize
\begin{equation}\label{cf}
\int_0^{T_b} e^{-rt} L(u(t))\, dt + e^{-rT_b} B\,,\end{equation}
subject to the dynamics
(\ref{detode}),
where the bankruptcy time $T_b$ is defined as in (\ref{Tb}),
while 
\begin{equation}\label{p1}
p(t)~=~\int^{T_b}_t(r+\lambda)e^{-(r+\lambda)s}ds+e^{(-r+\lambda)(T_b-t)}\cdot\theta(x^*)~=~1-(1-\theta(x^*))\, e^{-(r+\lambda)(T_b-t)}\,.
\end{equation}
}
\end{itemize}

Since in this case the optimal feedback control $u^*$  and the corresponding functions
$V,p$ may not be smooth, a concept of equilibrium solution should be 
more carefully defined.

\begin{definition}[Equilibrium solution in feedback form]\label{def:eqsol} 
 A couple of piecewise Lipschitz continuous  functions $u=u^*(x)$ and 
$p=p^*(x)$ provide an  equilibrium solution to the debt management problem (DMP), with 
continuous value function $V^*$, if
\begin{itemize}
\item[(i)] For every $x_0\in [0,x^*]$, ~
   $V^*$ is the minimum cost for the optimal control problem 
\begin{equation}\label{cost1}
\hbox{minimize:}~~\int_0^{T_b} e^{-rt} L(u(x(t)))\, dt + e^{-r T_b} B,
\end{equation}
subject to 
\begin{equation}\label{cont1}
\dot{x}(t)~=~\left(\dfrac{\lambda+r}{p^*(x(t))} -\lambda -\mu\right) x(t) - \dfrac{u(t)}{p^*(x(t))}\,,\qquad\qquad x(0)~=~x_0\,.
\end{equation}
Moreover, every Carath\'eodory solution of (\ref{cont1}) with $u(t)= u^*(x(t))$ is optimal.

\item[(ii)] 

For every $x_0\in [0,x^*]$, 
there exists at least one solution $t\mapsto x(t)$ of 
the Cauchy problem
\begin{equation}\label{fode}
\dot{x}~=~\left(\dfrac{\lambda+r}{p^*(x)} -\lambda -\mu\right) x - \dfrac{u^*(x)}{p^*(x)}\,,\qquad\qquad x(0)~=~x_0,
\end{equation}
such that
\begin{equation}\label{pB=}p^*(x_0)~=~\int^{T_b}_{0}(r+\lambda)e^{-(r+\lambda)t}dt+
e^{(-r+\lambda)T_b}\,\theta(x^*)~=~1-(1-\theta(x^*))\cdot e^{-(r+\lambda)T_b}\,,\end{equation}
with $T_b$ as in (\ref{Tb}).
\end{itemize}
\end{definition}


In the deterministic case, (\ref{ode0}) takes the form
\begin{equation}\label{ode1}
\begin{cases}
rV(x)&= ~H\bigl(x,V'(x),p(x)\bigr) \,,\\[4mm]
(r+\lambda)(p(x)-1)&=~H_\xi\bigl(x,V'(x),p(x)\bigr)\,p'(x)\,,\end{cases}\end{equation}
with Hamiltonian function (see Fig.~\ref{f:g53})
\begin{equation}\label{H2}
H(x,\xi,p)~=~\min_{\omega\in [0,1]}~\left\{ L(\omega) - \frac{\xi}{p} \omega \right\}+\left(\dfrac{\lambda+r}{p} -(\lambda +\mu)\,
\right) x\, \xi\,.
\end{equation}
We consider  solutions to (\ref{ode1}) with the boundary condition
\begin{equation}\label{bdc1}
\begin{cases}V(0)&=~0,\\ \\ V(x^*)&=~B,\end{cases}\qquad\qquad
\begin{cases}p(0)&=~1,\\ \\ p(x^*)&=~\theta(x^*)\,.\end{cases}
\end{equation}

Introduce the function
$$
H^{\max}(x,p)~\doteq~\sup_{\xi\geq 0}~H(x,\xi,p)~=~L\Big((\lambda+r)x-(\lambda+\mu)px\Big),
$$
with the understanding that 
\begin{equation}\label{Hmaxi}
H^{\max}(x,p)~=~+\infty\qquad\hbox{if}\qquad 
(\lambda+r)x-(\lambda+\mu)px~\geq~1\,.\end{equation}
If $(\lambda+r)x-(\lambda+\mu)px~<~1$, recalling (\ref{detode})  we define
\begin{equation}\label{usharp}u^\sharp(x,p)~=~(\lambda+r)x-(\lambda+\mu)px\,.\end{equation}
Notice that $u^\sharp$ is the control that keeps the debt $x$ constant in time. 
This value $u^\sharp$ achieves the minimum in (\ref{H2}) when
$$L'\Big((\lambda+r)x-(\lambda+\mu)px\Big) ~=~\frac{\xi}{p}\,.$$
This motivates the definition
\begin{equation}\label{xisharp} \xi^{\sharp}(x,p)~\doteq~\underset{\xi\geq 0}{\mathrm{argmax}}~H(x,\xi,p)
~=~p\,L'\Big((\lambda+r)x-(\lambda+\mu)px\Big).
\end{equation}
On the other hand, if $(\lambda+r)x-(\lambda+\mu)px~\geq~1$, 
then the function 
$\xi\mapsto H(x,\xi,p)$ is monotone increasing 
and we define $ \xi^{\sharp}(x,p)\doteq +\infty$.
Observe that
\begin{equation}\label{H-xi}
H_{\xi\xi}(x,\xi,p)~\leq~0\,,\qquad\qquad
\begin{cases}
H_{\xi}(x,\xi,p)&> ~0\qquad\forall 0\leq\xi<\xi^{\sharp}(x,p) \,,\\[4mm]
H_{\xi}(x,\xi,p)&<~0\qquad\forall \xi>\xi^{\sharp}(x,p)\,.
\end{cases}
\end{equation}
We regard the first equation in (\ref{ode1}) as an implicit ODE for the function $V$. For every $x\geq 0$ and $p\in [0,1]$, if $rV(x)>H^{\max}(x,p)$, then this equation has no solution. On the other hand, when
\[
0~\leq~r V(x)~\leq~H^{\max}(x,p),
\]
the implicit ODE (\ref{ode1})  can equivalently be written as
a differential inclusion (Fig.~\ref{f:g53}):
\begin{equation}\label{di1}
V'(x)~\in~\Big\{ F^-(x,V,p)\,,~F^+(x,V,p)\Big\}\,.\end{equation}

\begin{figure}
\centering
\includegraphics[scale=0.45]{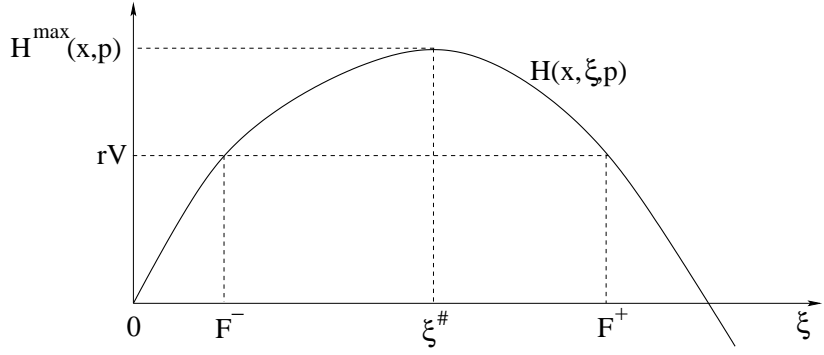}
\label{f:g53}
\caption{\small In the case where
$(\lambda+r)x- (\lambda+\mu)px <1$, the the 
Hamiltonian function $\xi\mapsto H(x,\xi,p)$ has a global maximum
$H^{max}(x,p)$. For $rV\leq H^{max}$, the values
$F^-(x,V, p)\leq \xi^\sharp(x,p)\leq F^+(x,V,p)$ 
are well defined.
}
\end{figure}

\begin{remark} 
Recalling  (\ref{detode}), we observe that
\begin{itemize}
\item 
The value $V'=F^+(x,V,p)\geq \xi^\sharp(x,p)$ corresponds to the choice of an optimal control such that
$\dot x<0$. 
\item The value $V'=F^-(x,V,p)\leq \xi^\sharp(x,p)$ corresponds to the choice of an optimal control such that
$\dot x>0$.
\item When $rV= H^{max}(x,p)$, then 
the value $V'=F^+(x,V,p)=F^-(x,V,p)= \xi^\sharp(x,p)$ corresponds to the unique control such that
$\dot x=0$. 
\end{itemize}
\end{remark}

Since $\xi\mapsto H(x,\xi,p)$ is concave down, the functions $F^\pm$ satisfy the following monotonicity properties (Fig.~\ref{f:g53})
\begin{itemize}
\item[{\bf (MP)}] {\it 
For any fixed $x,p$, the map $V\mapsto F^+(x,V,p)$ is decreasing, 
while $V\mapsto F^-(x,V,p)$ is increasing.}
\end{itemize}


For $V'~=~F^-$, the second ODE in (\ref{ode1}) can be written as 
\[
p'(x)~=~G^-\bigl(x,V(x),p(x)\bigr),
\]
where 
\begin{equation}\label{G-}
G^-(x,V,p)~\doteq~\dfrac{(r+\lambda)(p-1)}{H_{\xi}\bigl(x,F^-(x,V,p),p\bigr)}~\leq~0\,.
\end{equation}
\subsection{Construction of a solution.}
Consider the function
\begin{equation}\label{Wdef} W(x)~\doteq~\frac{1}{r}L\bigl((r-\mu)x\bigr),\end{equation}
with the understanding that $W(x)=+\infty$ if $(r-\mu)x\geq 1$.
Notice that $W(x)$ is the total cost of keeping the debt constantly
equal to $x$ 
(in which case there would be no bankruptcy and hence $p\equiv 1$).

Moreover,
denote by $(V_B(\cdot),p_B(\cdot))$  
the solution to the system of ODEs
\begin{equation}\label{ode2}
\begin{cases}
V'(x)&= ~F^-(x,V(x),p(x)) \,,\\[4mm]
p'(x)&=~G^{-}(x,V(x),p(x))\,,
\end{cases}
\end{equation}
with terminal conditions 
\begin{equation}\label{td2}
V(x^*)~=~B,\qquad\qquad p(x^*)~=~\theta(x^*)\,.
\end{equation}
Next, consider the point
\begin{equation}\label{x1}
x_1~\doteq~\inf \Big\{ x\in [0, x^*]\,;~~V_B(x)<W(x)\Big\},\end{equation}
and call $V_1(\cdot)$ the solution to the backward Cauchy problem
\begin{equation}\label{V1ode}
\begin{cases}
V'(x)&=~F^-(x, V(x), 1)\,,\qquad\qquad x\in [0, x_1],\\[3mm]
V(x_1)&=~W(x_1).
\end{cases}
\end{equation}
We will show that a feedback equilibrium solution to the 
debt management problem is obtained as follows (see Fig.~\ref{f:g55}).
\begin{align}
V^*(x)&=\begin{cases}V_1(x)\qquad \hbox{if}\quad x\in [0, x_1],\\ &\\
V_B(x)\qquad \hbox{if}\quad x\in [x_1, x^*].\end{cases}\label{sol1}\\
p^*(x)&=\begin{cases}1,&\textrm{ if }x\in [0, x_1],\\ &\\
p_B(x),&\textrm{ if }x\in \,]x_1, x^*].\end{cases}\label{sol2}\\
u^*(x)&=\begin{cases}\displaystyle\underset{\omega\in [0,1]}{\mathrm{argmin}}
\left\{L(\omega) - \frac{(V^*)'(x)}{p^*(x)}\omega\right\},&\textrm{ if }x\ne x_1,\\ &\\
(r-\mu)x_1,&\textrm{ if }x= x_1\,.\end{cases}\label{sol3}
\end{align}

\begin{figure}
\centering
\includegraphics[scale=0.45]{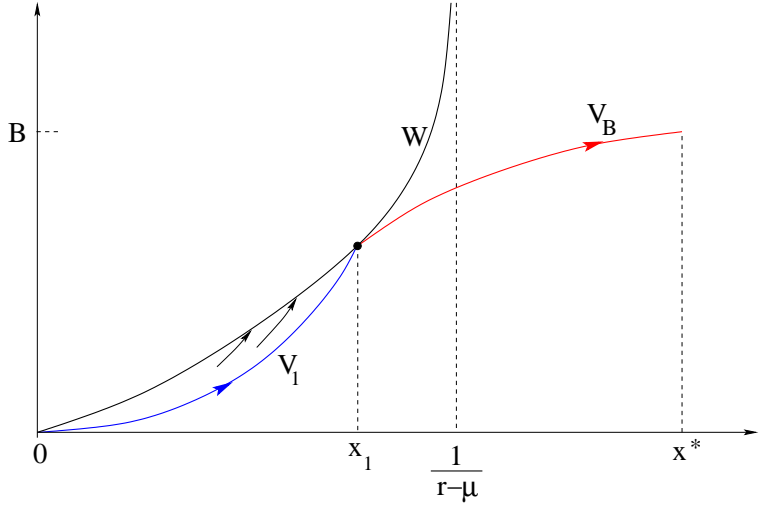}
\label{f:g55}
\caption{\small Constructing the equilibrium solution 
in feedback form. For an initial value of the debt $x(0)\leq x_1$,
the debt increases until it reaches $x_1$, then it is held at the 
constant value $x_1$. If the initial debt is $x(0)> x_1$, the debt 
keeps increasing until it reaches bankruptcy in finite time.}
\end{figure}

\begin{theorem}\label{t:41}
Assume that the cost function $L$ satisfies the assumptions {\bf (A)}, and moreover $L((r-\mu) x^*)>rB$.
Then the functions $V^*, p^*, u^*$ in (\ref{sol1})--(\ref{sol3}) are well defined, and provide an equilibrium solution to the
debt management problem, in feedback form. 
\end{theorem}
\begin{proof}
\begin{enumerate}[leftmargin=*,label=\textbf{\arabic*.}]
\item[]
\item The solution of (\ref{ode2})-(\ref{td2}) satisfies
the obvious bounds
\[V'\geq 0,\qquad p'\leq 0,\qquad V(x)~\leq ~B, \qquad  p(x)\in [\theta(x^*),1].\]
We begin by proving that the function $V_B$ 
is well defined
and strictly positive for $x\in \,]x_1, x^*]$.

To prove that 
$$V_B(x)~>~0\quad\forall x\in ]x_1,x^*]\,,$$
assume, on the contrary, that $V_B(y)=0$ for some $y>x_1\geq 0$.
By the properties of the function $F^-$ (see Fig.~\ref{f:g53}) it
follows
\begin{equation}\label{FLIP}F^-(x,V,p)~\leq~C_2 V\,,\end{equation}
for some constant $C_2$ and all $x\in [y, x^*]$, $p\in [\theta(x^*),1]$.   Hence, for any solution of (\ref{ode2}), 
 $V(y)=0$ implies $V(x)= 0$ for all $x\geq y$,
providing a contradiction.

Next, observe that the functions $F^-, G^-$ are locally Lipschitz continuous
as long as $0\leq V<H^{max}(x,p)$.
Moreover, $V(x)<W(x)$ implies
$$V(x)~<~W(x)~=~H^{max}(x,1)~\leq~H^{max}(x, p(x)).$$
Therefore, the functions $V_B, p_B$ are well defined on the 
interval $[x_1, x^*]$.   
\item[]
\item If $x_1=0$ the construction of the functions $V^*, p^*, u^*$
is already completed in step {\bf 1.}.    
In the case where $x_1>0$, we claim that
the function $V_1$ is well defined and satisfies
\begin{equation}\label{V1prop}
0~< ~V_1(x)~<~ W(x)\qquad\qquad\hbox{for} ~0<x<x_1\,.\end{equation}
Indeed, if $V_1(y)=0$ for some $y>0$, the Lipschitz property 
(\ref{FLIP}) again implies that $V_1(x)=0$ for all $x\geq y$.
This contradicts the terminal condition in (\ref{V1ode}).

To complete the proof of our claim (\ref{V1prop}), 
it suffices to show that 
\begin{equation}\label{W'}W'(x)~<~F^-(x, W(x), 1)\qquad\qquad\forall x\in \,]0, x_1].\end{equation}
This is true because
\begin{align*}
W'(x)&=\displaystyle~\frac{r-\mu}{r}\, L'\bigl(r-\mu)x\bigr)~=~\frac{r-\mu}{r} \,\xi^\sharp(x,1)~<~\xi^\sharp(x,1)\\[3mm]
&\displaystyle=~F^-\bigl(x,H^{max}(x,1), 1\bigr)~=~F^-(x,W(x),1). 
\end{align*}
\item[]
\item In the remaining steps,  we show that $V^*, p^*, u^*$ provide an equilibrium solution. Namely, they satisfy the properties (i)-(ii) 
in Definition \ref{def:eqsol}.   

To prove (i), call $V(\cdot)$ the value function 
for the optimal control problem (\ref{cost1})-(\ref{cont1}). 
 
For any initial value,  $x(0)=x_0$, in both cases $x_0\in [0, x_1]$ and $x_0\in \,] x_1, x^*]$, 
the feedback control $u^*$ in  (\ref{sol3})   
yields the cost $V^*(x_0)$. This implies
\[
V(x_0)~\leq ~V^*(x_0)\,.
\]  
To prove the converse inequality we need to show that, for any measurable control $u:[0,+\infty[\,\mapsto [0,1]$, calling $t\mapsto x(t)$ the solution to 
\begin{equation}\label{ode7}
\dot{x}~=~\left(\dfrac{\lambda+r}{p_{x_1}(x)} -\lambda -\mu\right) x - \dfrac{u(t)}{p_{x_1}(x)}\,,\qquad\qquad x(0)~=~x_0,
\end{equation}
one has
\begin{equation}\label{upVx_1}
\int_{0}^{T_b} e^{-rt} L(u(t))dt + e^{-rT_b}B~\geq~V^*(x_0),
\end{equation}
where 
$$T_b
~=~\inf\, \bigl\{t\geq 0\,;~~x(t)=x^*\bigr\}$$ is the bankruptcy time
(possibly with $T_b
=+\infty$).

For $t\in [0, T_b
]$, consider the
absolutely continuous function 
\[\phi^{u}(t)~\doteq~ \int_{0}^te^{-rs} L(u(s))ds + e^{-rt}V^*
(x(t)).\]
At any Lebesgue point $t$ of $u(\cdot)$, we compute
\begin{align*}
&\displaystyle \frac{d}{dt}\phi^{u}(t)~=~e^{-rt} \Big[L(u(t))-rV^*(x(t))+
(V^*)'(x(t))\cdot\dot{x}(t)\Big]\\[3mm]
&=\displaystyle ~e^{-rt}\left[L(u(t))-rV^*(x(t
))+(V^*)'(x(t))\left(\left(\frac{\lambda+r}{p^*(x(t))}-\lambda-\mu\right)x(t)-\frac{u(t)}{p^*(x(t))}\right)\right]\\[3mm]
&\displaystyle\ge~ e^{-rt}\left[ \min_{\omega\in[0,1]}
\left\{L(\omega)-\frac{(V^*)'(x(t))}{p^*(x(t))}\,\omega\right\} 
+ \left(\frac{\lambda+r}{p^*(x(t))}-\lambda-\mu\right)
x(t)(V^*)'(x(t))-rV^*(x(t))\right]\\[3mm]
&= ~e^{-rt}\Big[H\left(x(t),(V^*)'(x(t)), p^*(x(t))\right)-rV^*(x(t))\Big]~=~ 0.
\end{align*}
Therefore,
\[V^*(x_0)~=~\phi^{u}(0)~\leq~ \lim_{t
\rightarrow T_b-
}\phi^{u}(t
)~=~\int_{0}^{T_b
} e^{-rt} L(u(t))dt + e^{-rT_b
}B,\]
proving (\ref{upVx_1}).
\item[]
\item It remains to check (ii).  The case $x_0=0$ is trivial.
Two main cases will be considered.

{\it CASE 1: $x_0\in \,]0,x_1]$.} ~
Then there exists a solution $t\mapsto x(t)$
of (\ref{fode}) such that 
$p(t)=1$ and $x(t)\in \,]0,x_1]$ for all $t>0$. Moreover, 
\[
\lim_{t\to +\infty}~x(t)~=~x_1\,.
\]
In this case, $T_b=+\infty$ and   (\ref{pB=}) holds.
\v
{\it CASE 2:  $x_0\in ]x_1,x^*]$.}  ~ Then $x(t)>x_1$ for all 
$t\in [0,T_b]$. This implies
\[
\dot{x}(t)~=~H_{\xi}(x(t),V_B(x(t)),p_B(x(t)))\,.
\]
From the second equation in (\ref{ode1}) it follows 
\[\dot p(t)~=~p'(x(t))\dot x(t)~=~(r+\lambda)(p(t)-1),\]
\[
\frac{d}{dt} \ln (1-p(x(t)))~=~(r+\lambda)\,.
\]
Therefore, for every $t\in [0, T_b]$ one has
\[
p(x(0))~=~1-(1-p(x(t)))\cdot e^{-(r+\lambda)t}.
\]
Letting $t\to T_b$ we obtain
\[p(x_0)~=~1-(1-\theta(x^*)))\cdot e^{-(r+\lambda)T_b},\]
proving (\ref{pB=}).
\end{enumerate}
\end{proof}

\begin{remark}
The construction described at (\ref{sol1})--(\ref{sol3}) uniquely determines a feedback equilibrium  solution to the debt management problem.  

In general, however, we cannot rule out the possibility that a second solution exists. Indeed, if the solution
$V_B, p_B$ of (\ref{ode2})-(\ref{td2}) can be prolonged backwards to the entire interval $[0, x^*]$, then we could replace (\ref{sol1})-(\ref{sol2}) 
simply by $V^*(x)=V_B(x)$, $p^*(x)= p_B(x)$ for all $x\in [0, x^*]$. This would yield a second solution.

We claim that no other solutions can exist.   
This is based on the fact that the graphs of $W $ and 
$V_B$ cannot have any  other intersection, in addition to $0$ and $ x_1$.
Indeed, assume on the contrary that $W(x_2)= V_B(x_2)$ for some
$0<x_2<x_1$ (see Fig.~\ref{f:g56}).  Since $p_B(x_2)<1$ and
$W'(x_2)\leq V_B'(x_2)$, the inequalities
$$rW(x_2)~=~H(x_2, W'(x_2),1)~<~H(x_2, W'(x_2), p_B(x_2))~\leq~
H(x_2, V'_B(x_2), p_B(x_2))~=~rV_B(x_2)$$
yield a contradiction.

Next,
let $V^\dagger, p^\dagger$ be 
any equilibrium solution and define
$$x^\dagger~\doteq~\sup\bigl\{x\in [0,x^*]\,;~~p(x)=1\bigr\}.$$
Then
\begin{itemize}
\item On $]x^\dagger, x^*]$ the functions $V^\dagger, p^\dagger $ 
provide a solution to the backward Cauchy problem (\ref{ode2})-(\ref{td2}).

\item On $]0, x^\dagger]$ the function $V^\dagger$ provides the value function
for the optimal control problem 
$$\hbox{minimize:}~\int_0^\infty e^{-rt} L(u(t))\, dt$$
subject to the dynamics (with $p\equiv 1$)
$$\dot x ~=~(r-\mu) x - u\,,$$
and the state constraint $x(t)\in [0, x^\dagger]$ for all $t\geq 0$.
\end{itemize}
The above implies
\[\begin{cases}
V^\dagger(x)~&= ~V_B(x), \textrm{ if } x\in [x^\dagger,x^*],\\[3mm]
V^\dagger(x)~&\leq ~W(x), \textrm{ if } x\in [0, x^\dagger].
\end{cases}\]
Since $V^\dagger$ must be continuous at 
the point $x_2$, by the previous analysis this is possible only 
if $x_2=0$ or $x_2=x_1$.
\end{remark}

\begin{figure}
\centering
\includegraphics[scale=0.45]{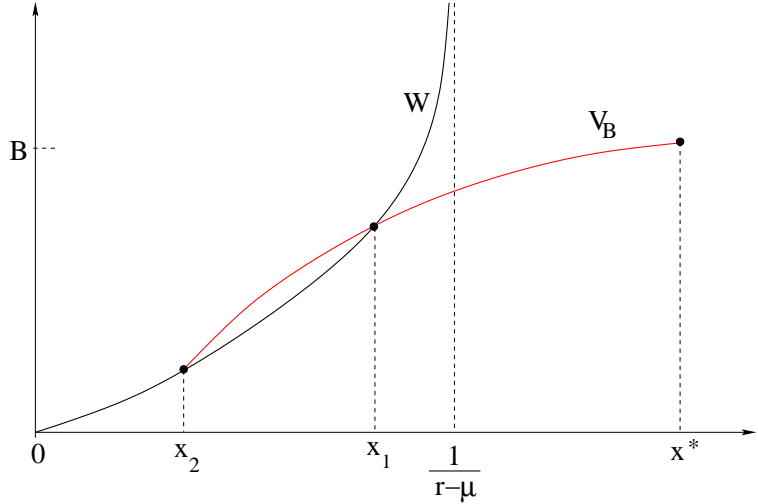}
\label{f:g56}
\caption{\small By the monotonicity properties of the Hamiltonian function $H$ in (\ref{H2}), the graphs of $V_B$ and $W$ cannot have 
a second intersection at a point $x_2>0$.}
\end{figure}

\section{Dependence on the bankruptcy threshold $x^*$.}

In this section we study the behavior of the value function $V_B$ 
when the maximum size $x^*$ of the debt, at which bankruptcy is declared,
becomes very large. 

For a given  $x^*>0$, we denote by $V_B(\cdot,x^*)$, $p_B(\cdot, x^*)$ 
the solution to the system (\ref{ode2})  with terminal data
(\ref{td2}).  
 Letting $x^*\to \infty$, we wish to understand whether the value function $V_B$  remains positive, or 
 approaches zero  uniformly on bounded sets. 
Toward this goal, we introduce the constant
\begin{equation}\label{M1def}
M_1~\doteq~\frac{2rB}{(r-\mu)L'(0)}\,.\end{equation}
From (\ref{ode1}) and (\ref{H2}) it follows 
\[
\frac{V_B'(x, x^*)}{p_B(x, x^*)}~\leq~\frac{rV_B(x,x^*)}{(r-\mu)x-1}\,.
\]
In turn, if $x^*> M_1$, this implies 
\[
\frac{V'_B(x, x^*)}{p_B(x, x^*)}~\leq~L'(0),\qquad\forall x\in [M_1, x^*]\,.
\]
Calling 
$u=u^*(x)$ the optimal feedback control, by (\ref{u^*}) we have
\begin{equation}\label{zero-u}
u^*(x)~=~0,\qquad\forall x\in \bigl[M_1,\,x^*\bigr]\,.
\end{equation}
In this case, the Hamiltonian function takes a simpler form, namely 
\begin{align*}
H(x,V',p)&=~\bigl[(\lambda+r)-(\lambda+\mu)p\bigr]\cdot \frac{V'x}{p}\,,\\
H_{\xi}(x,V',p)&=~\bigl[(\lambda+r)-(\lambda+\mu)p\bigr]x\,.
\end{align*}
Therefore, the system of ODEs (\ref{ode2}) can be written as
\begin{equation}\label{ode3}
\begin{cases}
V'&=~\dfrac{rp}{[(\lambda+r)-(\lambda+\mu)p]x}\, V \,,\\[6mm]
p'&=~(\lambda+r)\cdot \dfrac{p(p-1)}{[(\lambda+r)-(\lambda+\mu)p]\, x }\,.
\end{cases} 
\end{equation}

The second ODE of in (\ref{ode3}) is equivalent to 
\[\dfrac{d}{dx}\ln\Big(\dfrac{(1-p(x))^{r-\mu}}{p(x)^{r+\lambda}}\Big)~=~
\dfrac{r+\lambda}{x}.\]
Solving backward the above ODE with the terminal data $p(x^*)=\theta(x^*)$, 
we obtain
\begin{equation}\label{Im-p1}
p_B(x,x^*)~=~\dfrac{\theta(x^*) x^*}{x}\cdot \left(\dfrac{1-p_B(x,x^*)}{1-\theta(x^*)}\right)^{\frac{r-\mu}{r+\lambda}}\qquad\forall x\in \bigl[M_1,\,x^*\bigr]\,.
\end{equation}

Therefore,
\begin{equation}\label{lowbp}
p_B(x,x^*)~\geq~\dfrac{\left(\dfrac{\theta(x^*)x^*}{x}\right)^{\frac{r+\lambda}{r-\mu}}}{ 1+ \left(\dfrac{\theta(x^*)x^*}{x}\right)^{\frac{r+\lambda}{r-\mu}}}\qquad\forall x\in \bigl[M_1,\,x^*\bigr]\,.
\end{equation}

Different cases will be considered, depending on the properties of 
the function $\theta(\cdot)$.   By obvious 
modeling considerations, we shall always assume
$$\theta(x)\in [0,1],\qquad\qquad \theta'(x)\leq x\qquad\qquad\forall 
x\geq 0.$$
We first study the case where $\theta$ has compact support.  Recall that $M_1$ is the constant in (\ref{M1def}).

\begin{lemma} Assume that 
\begin{equation}\label{a-theta-1}
\theta(x)~=~0\qquad\forall x\geq M_2\,,
\end{equation}
for some constant $M_2\geq M_1$.  Then, for 
any $x^*>M_2$, the solution $V_B(\cdot, x^*)$, $p_B(\cdot, x^*)$ of 
(\ref{ode2})-(\ref{td2}) satisfies
\[V_B(x,x^*)~=~B\qquad\mathrm{and}\qquad p_{B}(x,x^*)~=~0 \qquad\forall x\in \bigl[M_2,\,x^*\bigr]\,.\]
\end{lemma}
\begin{proof}
By (\ref{Im-p1}) and (\ref{a-theta-1}), 
for every $x^*>M_2$ one has
\[p_{B}(x,x^*)~=0~\qquad\forall x\in \bigl[M_2,\,x^*\bigr]\,.\]
Inserting this into the first ODE in (\ref{ode3}), we obtain
\[V_B'(x,x^*)~=~0.\]
In turn, this yields 
$V_B(x,x^*)~=~B$ for all $ x\in \bigl[M_2,\,x^*\bigr]$.
This means that bankruptcy instantly occurs if the debt reaches $M_2$.
\end{proof}

\medskip

Next, we now consider that case where $\theta(x)>0$ for all $x$.
\begin{equation}\label{a-theta-2}
\theta(x)~>~0\qquad\forall x\in [0,\infty[\,.
\end{equation}
\begin{lemma}
If $x^*>M_1$ and $\theta(x^*)>0$, then  
\begin{equation}\label{ubV1}
V_B(x,x^*)~=~B\cdot \left(\frac{p_B(x,x^*)x}{\theta(x^*)x^*}\right)^{\frac{r}{r-\mu}}\qquad\forall x\in \bigl[M_1,\,x^*\bigr]\,.
\end{equation}
In particular, for $x\in \bigl[M_1,\,x^*\bigr]$ one has
\begin{equation}\label{bound-V}
B\cdot \left(1+\left(\dfrac{\theta(x^*)x^*}{x}\right)^\frac{r+\lambda}{r-\mu}\right)^{-\frac{r}{r+\lambda}}
~\leq~V_B(x, x^*)~\leq~B\cdot \Big(\frac{x}{\theta(x^*)x^*}\Big)^{\frac{r}{r-\mu}}\,.\end{equation}
\end{lemma}
\begin{proof}
Observe that $x\mapsto p_B(x, x^*)$ is a strictly decreasing function of $x$. 
For a fixed value of $x^*$, let $p\mapsto \chi(p):[\theta(x^*),1[\,\mapsto [0,x^*]$ 
be the inverse function of $p_B(\cdot, x^*)$. 
From (\ref{ode3}), a direct computation yields
\begin{equation}\label{odep}
\begin{cases}
\dfrac{d}{dp}V_B(\chi(p), x^*)&=\dfrac{rp}{[(\lambda+r)-(\lambda+\mu)p]\,\chi(p)}\cdot V_B(\chi(p),x^*)\cdot \chi'(p) \,,\\[6mm]
\dfrac{d}{dp}p_B(\chi(p),x^*)&=(\lambda+r)\cdot \dfrac{p(p-1)}{[(\lambda+r)-(\lambda+\mu)p]\cdot \chi(p) }\cdot \chi'(p)=1\,.
\end{cases} 
\end{equation}
From (\ref{odep}) it follows
\[
\dfrac{d}{dp}~\ln V_B(\chi(p),x^*) ~=~\frac{r}{\lambda+r}\cdot \frac{1}{p-1}\,.
\]
Solving the above ODE with the terminal data $V_B(x^*,x^*)=B$, 
$p_B(x^*,x^*)=\theta(x^*)$,
we obtain
\begin{equation}\label{Im-V1}
V_B(\chi(p), x^*)~=~
\left(\frac{1-p}{1-\theta(x^*)}\right)^{\frac{r}{r+\lambda}}B\,,
\end{equation}
hence
\[
V_B(x,x^*)~=~\left(\frac{1-p_B(x,x^*)}{1-\theta(x^*)}\right)^{\frac{r}{r+\lambda}}B.
\]
Recalling (\ref{Im-p1}), a direct computation yields (\ref{ubV1}). 
The upper and lower bounds for $V_B(x,x^*)$ in (\ref{bound-V}) now 
follow from (\ref{lowbp}) and the inequality 
$p_B(x,x^*)\leq 1$.
\end{proof}

\begin{corollary}\label{cor:infty-th} 
Assume that 
\begin{equation}\label{infty-th}
\limsup_{x\to +\infty}~\theta(x)\, x~=~+\infty.
\end{equation}

Then the value function $V^*= V^*(x, x^*)$ satisfies
\begin{equation}\label{Vto0}
\lim_{x^*\to +\infty} V(x, x^*)~=~0\qquad\forall x\geq 0\,. \end{equation}
\end{corollary}

Indeed, for $x\geq M_1$ we have $V(x, x^*)= V_B(x, x^*)$, and 
(\ref{Vto0}) follows from the second inequality in (\ref{bound-V}).    When $x<M_1$, since 
the map $x\mapsto V(x, x^*)$ is nondecreasing, we have
\[0~\leq~\lim_{x^*\to\infty} V(x,x^*)~\leq~\lim_{x^*\to\infty} V(M_1,x^*)~=~0\,.\]

\begin{corollary}\label{cor:low-bV} Assume that 
\begin{equation}\label{b2}
R~\doteq~\limsup_{x\to+\infty}~\theta(x)\cdot x~<~+\infty.
\end{equation}
Then
\begin{equation}\label{low-bV}
V_B(x, x^*)~\geq~B\cdot \left(1+\left(\frac{R}{x}\right)^\frac{r+\lambda}{r-\mu}\right)^{-\frac{r}{r+\lambda}}\qquad\forall~ x^*>x>M_1\,.
\end{equation}
Moreover,  the followings holds.
\begin{itemize}
\item [(i)] If
\begin{equation}\label{ca1}
\dfrac{\theta'(x)}{\theta(x)}+\dfrac{1}{x}~\geq~0\qquad\hbox{and}\qquad \theta'(x)~\leq~0\qquad\hbox{for all}~~x>0
\end{equation}
then 
\begin{equation}\label{inf-V}
\inf_{x^*>0}\,V_B(x,x^*)~=~\lim_{x^*\to\infty} V_B(x,x^*)~>~0\qquad\forall x\geq M_1\,. 
\end{equation}

\item [(ii)] Assume that  there exist  $0<\delta<1$  such that 
\begin{equation}\label{ca2}
\delta\cdot \dfrac{\theta'(x)}{\theta(x)}+\dfrac{1}{x}~<~0
\end{equation}
for all $x$ sufficiently large.
Then, for each $x>M_1$, there exists an optimal value 
$x^*=x^*(x)$ such that 
\begin{equation}\label{minV}V_B(x, x^*(x))~=~
\inf_{x^*\geq 0}~V_B(x,x^*).
\end{equation}
\end{itemize}
\end{corollary}
\begin{proof}
It is clear that (\ref{low-bV}) is a consequence of (\ref{bound-V}) and (\ref{b2}). We only need to prove (i) and (ii). 
For a fixed  $x\ge M_1$, we consider the functions of the variable $x^*$ alone:
\[
Y(x^*)~\doteq~V_B(x,x^*),\qquad\qquad q(x^*)~\doteq~p_{B}(x,x^*).
\]
Using (\ref{ubV1}) and (\ref{Im-p1}), we obtain
\begin{equation}\label{Y'}
\dfrac{Y'(x^*)}{Y(x^*)}~=~{\frac{r}{r-\mu}}\cdot \left(\frac{q'(x^*)}{q(x^*)}-\Big[\frac{\theta'(x^*)}{\theta(x^*)}+\frac{1}{x^*}\Big]
\right)\,,
\end{equation}
and
\begin{equation}\label{q'}
\frac{q'(x^*)}{q(x^*)}~=~\frac{\theta'(x^*)x^*+\theta(x^*)}{\theta(x^*)x^*}+\frac{r-\mu}{r+\lambda}\cdot \Big(\frac{-q'(x^*)}{1-q(x^*)}+\frac{\theta'(x^*)}{1-\theta(x^*)}\Big)\,.
\end{equation}
This implies 
\begin{multline}\label{re1}
\frac{q'(x^*)}{q(x^*)}-\left[\frac{\theta'(x^*)}{\theta(x^*)}+\frac{1}{x^*}\right]~=~\left[\dfrac{1}{1+\frac{r-\mu}{r+\lambda}\cdot\frac{q(x^*)}{1-q(x^*)}}-1\right]\cdot\left[\frac{\theta'(x^*)}{\theta(x^*)}+\frac{1}{x^*}\right]
\\+\frac{r-\mu}{(r+\lambda) \left(1+\frac{r-\mu}{r+\lambda}\cdot\frac{q(x^*)}{1-q(x^*)}\right)}\cdot \frac{\theta'(x^*)}{1-\theta(x^*)}\,.
\end{multline}
If (\ref{ca1}) holds, then (\ref{Y'}) and (\ref{re1}) imply 
\[
\dfrac{Y'(x^*)}{Y(x^*)}~=~\frac{q'(x^*)}{q(x^*)}-\Big[\frac{\theta'(x^*)}{\theta(x^*)}+\frac{1}{x^*}\Big]~\leq~0\qquad
\hbox{for all} ~x^*>x\geq M_1.
\]
Hence the function $Y$ is non-increasing. This proves (\ref{inf-V}).
\v
To prove (ii), we observe that
\[\limsup_{x^*\to\infty}~\left(\dfrac{1}{1+\frac{r-\mu}{r+\lambda}\cdot\frac{q(x^*)}{1-q(x^*)}}-1\right)~<~0\,,\qquad\qquad \lim_{x^*\to\infty}\theta(x^*)~=~0\,.\]
Hence (\ref{ca2}) and (\ref{re1}) imply  
\[\frac{q'(x^*)}{q(x^*)}-\left[\frac{\theta'(x^*)}{\theta(x^*)}-\frac{1}{x^*}\right]~>~0,\]
for all $x^*$ sufficiently large. By (\ref{Y'}) this yields
\[\dfrac{Y'(x^*)}{Y(x^*)}~>~0\]
for all $x^*$ large enough.
Hence there exists some particular value $x^*(x)\geq x$ where 
the function $x^*\mapsto Y(x^*)= V_B(x, x^*)$ attains its
global minimum. This yields (\ref{minV}).
\end{proof}

\section{Dependence on $x^*$ in the stochastic case}

In this section we study how the value function depends on the
bankruptcy threshold $x^*$, in the stochastic case 
where $\sigma > 0$.
Extensions of  Corollaries \ref{cor:infty-th} and \ref{cor:low-bV}, 
will be proved,  constructing upper and lower bounds
for the solution $V(\cdot, x^*)$, $p(\cdot, x^*)$ of 
the system (\ref{ode0})-(\ref{bdc}), in the form
\begin{equation}\label{Vpb}
V_2(x)~\leq~V(x, x^*)~\leq ~V_1(x), \qquad\qquad 
p_1(x)~\leq~p(x, x^*)~\leq ~p_2(x).\end{equation}
\begin{enumerate}[leftmargin=*,label=\textbf{\arabic*.}]
\item[] 
\item We begin by constructing a suitable pair of functions $V_1,p_1$.
Let $(p_1,\Tilde{V}_1)$ be the  solution to the 
backward Cauchy problem
\begin{equation}\label{pV1}\begin{cases}
r\Tilde{V}_1(x)&= ~\displaystyle \Big(\frac{\lambda+r}{p_1}+\sigma^2\Big)x \Tilde{V}_1'\,, \\[6mm]
(r+\lambda)(p_1-1)&=~\displaystyle\Big(\frac{\lambda+r}{p_1}+\sigma^2\Big) x p'_1\,,\end{cases}
\qquad\qquad
\begin{cases}\Tilde{V}_1(x^*)&= ~B, \\[6mm]
p_1(x^*)&=~\theta(x^*).\end{cases}
\end{equation}
This solution  satisfies
\begin{equation}\label{p1s}
p_1(x)~=~\dfrac{\theta(x^*)x^*}{x}\cdot \left(\dfrac{1-p_1(x)}{1-\theta(x^*)}\right)^{\frac{\sigma^2+\lambda+r}{\lambda+r}},\qquad\qquad 
\lim_{x\to 0+} p_1(x)~=~1\,,
\end{equation}
\begin{equation}\label{V1}
\Tilde{V}_1(x)~=~B\cdot \left(\frac{1-p_1(x)}{1-\theta(x^*)}\right)^{\frac{r}{r+\lambda}},\qquad\qquad \lim_{x\to 0+} \tilde{V}_1(x)~=~0\,.
\end{equation}
Using (\ref{pV1}) and (\ref{p1s}) one obtains
\begin{align*}
-1&\displaystyle=~p'_1(x)\cdot\left(\frac{x}{p_1(x)}+\frac{\sigma^2+r+\lambda}{r+\lambda}\cdot\frac{x}{1-p_1(x)}\right)\\[4mm]
&\displaystyle=~p'_1(x) \cdot\left(\frac{x}{p_1(x)}+\frac{\sigma^2+r+\lambda}{r+\lambda}\cdot \frac{1-\theta(x^*)}{(\theta(x^*)x^*)^{\frac{r+\lambda}{r+\lambda+\sigma^2}}}\cdot \frac{x^{\frac{\sigma^2}{r+\lambda+\sigma^2}}}{p_1(x)^{\frac{\lambda+r}{\lambda+r+\sigma^2}}}\right)\,.
\end{align*}
Since $p_1$ is monotone decreasing, it follows that $p_1''(x)>0$ 
for all $x\in \,]0, x^*[\,$.  In turn, this yields
\[
(r+\lambda)(1-p_1)+\Big(\frac{\lambda+r}{p_1}+\sigma^2\Big)xp_1'+\frac{\sigma^2x^2}{2}p_1''~>~0\,.
\]
Recalling (\ref{Hb2}), we have 
\begin{equation}\label{sp1}
(r+\lambda)(1-p_1)+H_{\xi}(x,\xi,p_1)p_1'+\frac{\sigma^2x^2}{2}p_1''~>~0\qquad\qquad \forall \xi\geq 0.
\end{equation}
Next, differentiating  both sides of the first 
ODE in (\ref{pV1}), we obtain
\[
\left(r-\sigma^2-\frac{\lambda+r}{p_1} +\frac{(\lambda+r)p_1'}{p_1^2}x\right)\cdot \Tilde{V}_1'~=~\left(\frac{\lambda+r}{p_1}+\sigma^2\right)x \Tilde{V}_1''\qquad\forall x\in \,]0,x^*[\,.
\]
This implies 
\[\Tilde{V}_1''(x)~<~0\qquad\forall x\in \,]0,x^*[\,.\]
Recalling (\ref{Hb1}) and (\ref{pV1}), we obtain
\begin{equation}\label{sV1}
-r\Tilde{V}_1+H(x,\Tilde{V}_1',p_1)+\frac{\sigma^2x^2}{2}\Tilde{V}_1''~<~0\,.
\end{equation}
When  $x\geq \frac{1}{\lambda+r}$, 
the map $p\mapsto H(x,\xi,p)$ is monotone decreasing. 
Defining
\[
V_1(x)~\doteq~\begin{cases}
\Tilde
{V}(\frac{1}{r+\lambda}) \qquad &\hbox{for}~~ x\in \Big[0,\frac{1}{r+\lambda}\Big],\\[4mm]
\Tilde
{V}(x)\qquad&\hbox{for}~~ x\in \Big[\frac{1}{r+\lambda},x^*\Big]\,,\end{cases}
\]
we thus have 
\begin{equation}\label{sV1-m}
-rV_1(x)+H(x,V_1'(x),q)+\frac{\sigma^2x^2}{2}V_1''(x)~\leq~0\qquad\forall q\geq p_1(x)\,.
\end{equation}
\item[] 
\item We now construct the functions $V_2,p_2$. Defining
\[
\tilde
{p}_2(x)~\doteq~\frac{1}{x}
\left(\theta(x^*)x^*+\frac{2}{r-\mu}\right)\,,
\]
a straightforward computation yields  
\[
\tilde{p}_2'(x)~=~-\frac{\tilde{p}_2(x)}{x}~<~0\,,\qquad\qquad \tilde{p}_2''(x)~=~2\cdot \frac{\tilde{p}_2(x)}{x^2}\,.
\]
Set 
\begin{equation}\label{x*2}
x_2~\doteq~\theta(x^*)x^*+\frac{2}{r-\mu}\,,
\end{equation}
and consider the continuous function
\begin{equation}\label{p2}
p_2(x)~=~\min~\bigl\{ 1, \,\tilde p_2(x)\bigr\}.
\end{equation}
For $x\in [0,x_2[\,$ one has $p_2(x)=1$ and hence
\[
(r+\lambda)(1-p_2)+H_{\xi}(x,\xi,p_2)p_2'+\frac{\sigma^2x^2}{2}p_2''~=~0\,.
\]
On the other hand, for $x\in ]x_2,x^*[$ and $\xi\geq0$, one has $p_2(x)=
\tilde p_2(x)<1$, and
\begin{equation}\label{sig1}
H_{\xi}(x,\xi,p_2)~\geq~\frac{(\lambda+r) x-1}{p_2}+(\sigma^2-\lambda-\mu)x~\geq~(r-\mu)x_1-1~=~(r-\mu)\theta(x^*)x^*+1~\geq~0\,.
\end{equation}
Recalling (\ref{Hb2}), we get
\begin{align*}
(r+\lambda)(1-p_2)&+H_{\xi}(x,\xi,p_2)p_2'+\frac{\sigma^2x^2}{2}p_2''\\[4mm]
&\leq~(r+\lambda)(1-p_2)+\Big[\frac{(\lambda+r) x-1}{p_2}+(\sigma^2-\lambda-\mu)x\Big]p'_2(x)+\frac{\sigma^2x^2}{2}p_2''\\[4mm]
&=~(r+\lambda)(1-p_2)-\Big[\frac{(\lambda+r) x-1}{p_2}+(\sigma^2-\lambda-\mu)x\Big]\cdot \frac{p_2(x)}{x}+\sigma^2p_2\\[4mm]
&=~(r+\lambda)(1-p_2)-\Big[(\lambda+r)-\dfrac{1}{x}+(\sigma^2-\lambda-\mu)p_2\Big]+\sigma^2p_2\\[4mm]
&=~\dfrac{1}{x}-(r-\mu)p_2~=~ -\frac{(r-\mu)\theta(x^*)x^*}{x}-\dfrac{1}{x}~<~0\,.
\end{align*}
In particular, 
\begin{equation}\label{sp2}
(r+\lambda)(1-p_2)+H_{\xi}(x,\xi,p_2)\cdot p_2'+\frac{\sigma^2x^2}{2}p_2''~\leq~0\qquad\forall x\in \,]0,x^*[, ~\xi\geq 0\,.
\end{equation}

Next, define
\begin{equation}\label{V2}
V_2(x)~\doteq~(1-p_2(x))\,B\qquad\forall x\in [0,x^*]\,.
\end{equation}
For all $x\in [0,x_2]$, we thus have  $V_2(x)=0$, and hence
\begin{equation}\label{ine1}
-rV_2+H(x,V'_2,q)+\frac{\sigma^2x^2}{2}V_2''~=~H(x,0,q)~=~0\qquad\forall 
 q\in ]0,1]\,.
\end{equation}
On the other hand, for $x\in \,]x_2, x^*] $ we have
\[
V'_2(x)~=~B\cdot\frac{p_2(x)}{x}~>~0\qquad\mathrm{and}\qquad 
V''_2(x)~=~-2B\cdot \frac{p_2(x)}{x^2}\,.
\]
Recalling (\ref{Hb1}), (\ref{p2}), (\ref{sig1}) and (\ref{x*2}),  we estimate
\begin{align*}
-rV_2&+H(x,V_2',p_2)+\frac{\sigma^2x^2}{2} V_2''~\geq~-rV_2+\Big(\frac{(\lambda+r) x-1}{p_2}+(\sigma^2-\lambda-\mu)x\Big)V'_2+\frac{\sigma^2x^2}{2} V_2''\\[4mm]
&=B\cdot \Big[rp_2-r+\Big(\lambda+r-\dfrac{1}{x}+(\sigma^2-\lambda-\mu)p_2(x)\Big)-\sigma^2p_2\Big]\\[4mm]
&=B\cdot \Big(\lambda-\dfrac{1}{x}-(\lambda+\mu-r)p_2\Big)~=~B\cdot \Big[\lambda(1-p_2(x))+(r-\mu)p_2(x)-\frac{1}{x}\Big]~>~0
\end{align*}
for all $x\in \,]x_2,x^*[$. 

Recalling (\ref{x*2}), one has 
\[
(\lambda+r)x ~>~1\qquad\forall x\in ]x_2,x^*]\,.
\]
Therefore the map $p\to H(x,V'_2(x),p)$ is monotone decreasing  on $[0,1]$, for all $x\in  ]x_2,x^*]$. This implies
\[
-rV_2+H(x,V_2',q)+\frac{\sigma^2x^2}{2} V_2''~\geq~0\qquad\forall x\in \, ]x_2,x^*], 
~~q\in ]0,p_2(x)]\,.
\]
Together with (\ref{ine1}), we finally obtain
\begin{equation}\label{V2-s}
-rV_2(x)+H(x,V'_2(x),q)+\frac{\sigma^2x^2}{2} V_2''(x)~\geq~0\qquad\forall x\in ]0,x^*[\,, 
~~q\in ]0,p_2(x)]\,.
\end{equation}
\end{enumerate}

Relying on (\ref{sp1}), (\ref{sV1}), (\ref{sp2}) and (\ref{V2-s}), 
and using the same comparison argument as in the proof of 
Theorem~\ref{thm:existence} we now prove

\begin{theorem}\label{thm:6.1}
In addition to (A1), 
assume that $\sigma>0$ and $ \theta(x^*)>0$. 
Then the system  (\ref{ode0}) with boundary conditions 
(\ref{bdc}) admits a solution $({V}(\cdot,x^*),{p}(\cdot,x^*))$ 
satisfying the bounds (\ref{Vpb}) for all $x\in [0,x^*]$.
\end{theorem}
\begin{proof}
\begin{enumerate}[leftmargin=*,label=\textbf{\arabic*.}]
\item[] 
\item Recalling $\mathcal{D}$ in (\ref{Ddef}), we  claim that the domain
\begin{equation}\label{D1def}
\mathcal{D}^0~=~\Big\{ (V,p)\in\mathcal{D}~\Big|~(V(x),p(x))\in [V_2(x),V_1(x)]\times [p_1(x),p_2(x)],\quad\forall x\in [0,x^*] \Big\}
\end{equation}
 is positively invariant for  the semigroup 
$\{S_t\}_{t\geq 0}$, generated by the parabolic system 
(\ref{parV})-(\ref{parp}). Namely:
\[
S_{t}(\mathcal{D}^0)~\subseteq~\mathcal{D}^0\qquad\forall t\geq 0\,.
\]
Indeed, from the proof of Theorem \ref{thm:existence}, we have 
\begin{equation}\label{signs}
p_x(t,x)~\leq~0~\leq~V_x(t,x)\qquad\forall t>0,x\in ]0,x^*[\,.
\end{equation}
We now observe that 
\begin{itemize}
\item[(i)] For any $V(\cdot,\cdot)$ with $V_x\geq 0$, 
by (\ref{sp1}) the  function $p(t,x)= p_1(x)$ is a subsolution of the second equation 
in (\ref{par1}).    Similarly, by (\ref{sp2}), the function $p(t,x)= p_2(x)$
is a supersolution.
\item[(ii)] For any $p(\cdot,\cdot)$ with $p\in [0,1]$ and $p_x\leq 0$, 
by (\ref{sV1-m}) the function $V(t,x)= V_1(x)$ is a supersolution of
the first equation in (\ref{par1}).  
Similarly, by (\ref{V2-s}), the function $V(t,x)= V_2(x)$
is a subsolution.
\end{itemize}
Together, (i)-(ii) imply the positive invariance of the domain $\mathcal{D}^0$.
\item[] 
\item Using the same argument as in step {\bf 4} of the proof of Theorem~\ref{thm:existence}, we conclude that the system (\ref{ode0})-(\ref{bdc})
admits a solution $(V,P)\in \mathcal{D}^0$. 
\end{enumerate}
\end{proof}

\begin{corollary} Let the 
assumptions in Theorem~\ref{thm:6.1} hold.  If
\[\limsup_{s\to +\infty}~\theta(s)\, s~=~+\infty,\]
then, for all $x\geq 0$, the value function $V(\cdot, x^*)$
satisfies
\begin{equation}\label{V00}
\lim_{x^*\to\infty} V(x,x^*)~=~0.
\end{equation}
\end{corollary}
\begin{proof} 
Using (\ref{p1}), (\ref{V1}) 
and Theorem~\ref{thm:6.1}, we have the estimate
\[V(x,x^*)~\leq~V_1(x)~=~B\cdot \left(\frac{1-p_1(x)}{1-\theta(x^*)}\right)^{\frac{r}{r+\lambda}}~\leq~B\cdot \Big(\frac{x}{\theta(x^*)x^*}\Big)^{\frac{r}{r+\lambda+\sigma^2}}\,\]
for all $x\geq \frac{1}{r+\lambda}$. This implies that (\ref{V00}) holds for all $x\geq \frac{1}{r+\lambda}$. Since $x\mapsto V(x,x^*)$ is monotone increasing, we then have
\[0~\leq~\lim_{x^*\to\infty} V(x,x^*)~\leq~\lim_{x^*\to\infty}V\left(\frac{1}{r+\lambda}\,,~x^*\right)~=~0\qquad\forall ~x\in \Big[0,~ \frac{1}{r+\lambda}\Big].\]
This completes the proof of (\ref{V00}).
\end{proof}

\begin{corollary} Let the 
assumptions in Theorem~\ref{thm:6.1} hold.  If
\[C_1~\doteq~\limsup_{s\to+\infty}~\theta(s)\, s~<~+\infty,\]
then
\begin{equation}\label{Lowb-V}
\liminf_{x^*\to\infty}V(x, x^*)
~\geq~B\cdot \left(1-\frac{C_2}{x}\right)\qquad\forall~ x>M_2\,,
\end{equation}
where the constants $C_2,M_2$ are defined as
\[C_2~\doteq~C_1+\frac{2}{r-\mu}\qquad\mathrm{and}\qquad M_2~\doteq~\frac{\lambda+\mu-r}{\lambda}C_1+\frac{2\lambda+\mu-r}{\lambda (r-\mu)}+1\,.\]
\end{corollary}
\begin{proof}
This follows from (\ref{p2}), (\ref{V2}) and Theorem~\ref{thm:6.1}.
\end{proof}

\section{Concluding remarks}

 If we allow $x^*=+\infty$, then the equations 
(\ref{ode1}) admit the trivial solution
$V(x)=0$, $p(x)=1$, for all $x\geq 0$.  
This corresponds to a Ponzi scheme, producing a debt whose size grows exponentially, without bounds.  In practice, this is not realistic because there is a maximum amount  of liquidity 
that the market can supply. It is interesting to 
understand what happens when this bankruptcy threshold 
$x^*$ is very large.  
 
Our analysis has shown that three cases can arise, depending on 
the fraction $\theta$ of outstanding capital that lenders can recover.
\begin{itemize}
\item[(i)] If $\displaystyle\lim_{s\to+\infty} \theta(s)\, s= +\infty$, 
then it is convenient to choose $x^*$ as large as possible.
By delaying the time of bankruptcy, the expected 
cost for the borrower,
exponentially discounted in time, approaches zero.  

\item [(iii)] If $\displaystyle\lim_{s\to+\infty} \theta(s)\, s < +\infty$ and (\ref{ca1}) holds then it is still 
convenient to choose $x^*$ as large as possible. However,by delaying the time of bankruptcy, the cost to the borrower remains uniformly positive.

\item[(iii)] If $\displaystyle \lim_{s\to+\infty} \theta(s)\, s < +\infty$ and (\ref{ca2}) holds, then for every initial value $x_0$ 
of the debt there is a choice $x^*(x_0)$ of the bankruptcy threshold which is optimal for the borrower.
\end{itemize}

Examples corresponding to three cases (i)--(iii) are obtained
by taking
\begin{equation}\label{th4}
\theta(s)~=~\min\left\{1,\, \frac{R_0}{s^\alpha}
\right\}
\end{equation}
with $0<\alpha<1$,  $\alpha=1$, or $\alpha>1$, respectively.
\v
It is important to observe that the choice
of the optimal bankruptcy threshold $x^*$ is never ``time consistent".
Indeed, at the beginning of the game the borrower announces
that he will declare bankruptcy 
when the debt reaches size $x^*$.  Based on this information, the lenders determine the discounted price of bonds.
However, when the time $T_b$ comes when $x(T_b)=x^*$, 
it is never convenient for the borrower to declare bankruptcy.
It is  the creditors, or an external authority, that must enforce
termination of the game.   

To see this, assume that at time $T_b$ when $x(T_b)=x^*$ the borrower 
announces that he has changed his mind, and 
will declare bankruptcy only at 
the later time $T'_b$ when the debt reaches $x(T'_b)=2x^*$.   If he chooses a control $u(t)=0$ for $t>T_b$, his discounted cost 
will be
$$e^{-(T'_b-T_b)r} B~<~B.$$
This new strategy is thus  always convenient for the borrower.
On the other hand, it can be much worse for the lenders.
Indeed, consider an investor having a unit amount of outstanding capital at time $T_b$.   If bankruptcy is declared at time $T_b$, he will recover
the amount $\theta(x^*)$.  However, if bankruptcy is declared 
at the later time $T'_b$, his discounted payoff will be
$$
\int_{T_b}^{T'_b} (r+\lambda)e^{-(r+\lambda)(t-T_b)}\, dt +
e^{-(r+\lambda)(T'_b-T_b)}\theta(2x^*).$$
To appreciate the difference, consider the
deterministic case, assuming that
$\theta(\cdot)$ is the function in  (\ref{th4}), with $\alpha\geq 1$, and that 
$x^*\geq M_1$.   By the analysis at the beginning of Section~5, we have
$u^*(x)~=~0$ for all $ x\in [x^*,2x^*]$.
Replacing $x^*$ with $2x^*$ in (\ref{Im-p1}) we obtain  that the solution to (\ref{ode3})
with terminal data 
$$p(2x^*)~=~ \theta(2x^*) ~=~ \frac{R_0}{(2x^*)^\alpha}$$ satisfies
\[
p_B(x^*,2x^*)~=~2\theta(2x^*)\cdot \left(  \frac{1-p_B(x^*,2x^*)}{1-\theta(2x^*)}\right)^{\frac{r-\mu}{r+\lambda}}~<~2\theta(2x^*)~=~2^{1-\alpha}\theta(x^*)~\leq~\theta(x^*)\,.
\]
If the investors had known in advance that bankruptcy is declared at $x=2x^*$ (rather than at $x=x^*$),
the bonds would have fetched a smaller price.

In conclusion, if the bankruptcy threshold $x^*$ is chosen by the debtor,
the only Nash equilibrium can be $x^*=+\infty$.   In this case, 
the model still allows bankruptcy
to occur, when total debt approaches infinity in finite time.


\begin{thebibliography}{6111}

\bibitem{AG} M.~Aguiar and G. Gopinath,
Defaultable debt, interest rates and the current account.
{\it  J. International Economics} {\bf  69} (2006), 64--83.

\bibitem{A}
H.~Amann, 
Invariant sets and existence theorems for semilinear parabolic equation and elliptic system, {\it J. Math. Anal. Appl.} {\bf 65}, 432--467, 
1978.

\bibitem{AR} C.~Arellano and A.~Ramanarayanan,
Default and the maturity structure in sovereign
bonds.  {\it J. Political Economy} {\bf 120}  (2102), 187--232.

 \bibitem{BCD}
 M.~Bardi and I.~Capuzzo Dolcetta, {\it
 Optimal Control and Viscosity
 Solutions of Hamilton-Jacobi-Bellman Equations}, Birkh\"auser, 1997.

\bibitem{BO} T.~Basar and G.~J.~Olsder, {\it Dynamic Noncooperative
 Game Theory}, $2^d$ Edition, Academic Press, London 1995.
\bibitem{B}  A.~Bressan,
Noncooperative differential games.  {\it Milan J.~Math.}
 {\bf 79} (2011), 357--427.

\bibitem{BJ}
A.~Bressan and Y.~Jiang,
 Optimal open-loop strategies in a debt management problem, 
 submitted.
 

\bibitem{BN}
A.~Bressan and Khai T. Nguyen,
 An equilibrium model of debt and bankruptcy, {\it ESAIM; Control, Optim. Calc. Var.}, to appear.


\bibitem{BPi}
A.~Bressan and B.~Piccoli,  {\it Introduction to the Mathematical
Theory of Control},
AIMS Series in Applied Mathematics, Springfield Mo. 2007.


\bibitem{MP}  M.~Burke and K.~Prasad, 
An evolutionary model of debt.
{\it J.~Monetary Economics} {\bf 49} (2002) 1407-1438.

\bibitem{C}
G.~Calvo,  Servicing the Public Debt: The Role of Expectations.
{\it American Economic
Review}  {\bf 78} (1988), 647--661.

\bibitem{EG} 
J.~Eaton and M. Gersovitz, Debt with potential repudiation: 
Theoretical and empirical analysis. {\it Rev. Economic Studies} {\bf 48} (1981), 289--309.


\bibitem{NT} G.~Nu$\tilde{\rm n}$o and C.~Thomas,
Monetary policy and sovereign debt vulnerability, 
Working document n.~1517, Banco de Espa\~{n}a Publications, 2015.

\bibitem{O} 
B.~Oksendahl,  {\it Stochastic Differential Equations: An Introduction with Applications}.
Springer-Verlag, 2013.
 
\bibitem{Shreve}
S.~E.~Shreve, {\it Stochastic calculus for finance. II. 
Continuous-time models.} Springer-Verlag, New York, 2004.

\end{thebibliography}
\end{document}